\documentclass[11pt]{article}
\usepackage{amsmath, amssymb, amsthm}
\usepackage{verbatim}
\usepackage{multicol}
\usepackage{enumerate}
\usepackage{comment}
\usepackage[none]{hyphenat}
\usepackage{hyperref}
\hypersetup{
	colorlinks=true,
	linkcolor=blue,
	filecolor=magenta,
	urlcolor=cyan,
	citecolor=blue
}

\usepackage{pgf}
\usepackage{tikz}
\usepackage{ifthen}
\usetikzlibrary{math}
\usetikzlibrary{positioning,arrows,shapes,decorations.markings,decorations.pathreplacing,matrix,patterns}
\tikzstyle{vertex}=[circle,draw=black,fill=black,inner sep=0,minimum size=3pt,text=white,font=\footnotesize]

\date{}
\title{\vspace{-0.8cm}On the chromatic number of disjointness graphs of curves}
\author{
	J\'{a}nos Pach \thanks{\'{E}cole Polytechnique F\'{e}d\'{e}rale de Lausanne, Research partially supported by Swiss National Science Foundation grants no. 200020-162884 and 200021-175977. \emph{e-mail}: \textbf{\{janos.pach, istvan.tomon\}@epfl.ch}}
	\and
	Istv\'{a}n Tomon \footnotemark[1]	
}

\theoremstyle{plain}
\newtheorem{theorem}{Theorem}
\newtheorem{definition}[theorem]{Definition}
\newtheorem{corollary}[theorem]{Corollary}
\newtheorem{claim}[theorem]{Claim}
\newtheorem{lemma}[theorem]{Lemma}

\newtheorem{problem}[theorem]{Problem}

\theoremstyle{definition}

\DeclareMathOperator{\sg}{sg}

\begin{document}

\maketitle
\sloppy

\begin{abstract}
Let $\omega(G)$ and $\chi(G)$ denote the clique number and chromatic number of a graph $G$, respectively.
The {\em disjointness graph} of a family of curves (continuous arcs in the plane) is the graph whose vertices correspond to the curves and in which two vertices are joined by an edge if and only if the corresponding curves are disjoint. A curve is called {\em $x$-monotone} if every vertical line intersects it in at most one point. An $x$-monotone curve is {\em grounded} if its left endpoint lies on the $y$-axis.

We prove that if $G$ is the disjointness graph of a family of grounded $x$-monotone curves such that $\omega(G)=k$, then $\chi(G)\leq \binom{k+1}{2}$. If we only require that every curve is $x$-monotone and intersects the $y$-axis, then we have $\chi(G)\leq \frac{k+1}{2}\binom{k+2}{3}$. Both of these bounds are best possible. The construction showing the tightness of the last result settles a 25 years old problem: it yields that there exist $K_k$-free disjointness graphs of $x$-monotone curves such that any proper coloring of them uses at least $\Omega(k^{4})$ colors. This matches the upper bound up to a constant factor.
\end{abstract}

\section{Introduction}

Given a family of sets, $\mathcal{C}$, the \emph{intersection graph of $\mathcal{C}$} is the graph, whose vertices correspond to the elements of $\mathcal{C}$, and two vertices are joined by an edge if the corresponding sets have a nonempty intersection. Also, the \emph{disjointness graph of $\mathcal{C}$} is the complement of the intersection graph of $\mathcal{C}$, that is, two vertices are joined by an edge if the corresponding sets are disjoint. As usual, we denote the {\em clique number,} the {\em independence number}, and the {\em chromatic number} of a graph $G$ by $\omega(G), \alpha(G)$ and $\chi(G)$, respectively.

\paragraph{Clique number vs. chromatic number.} Computing these parameters for intersection graphs of various classes of geometric objects (segments, boxes, disks etc.) or for other geometrically defined graphs (such as visibility graphs) is a computationally hard problem and a classic topic in computational and combinatorial geometry \cite{AgMu, CaCa, FP11, Chaya, KrM, KrNe}. There are many interesting results connecting the clique number and the chromatic number of geometric intersection graphs, starting with a beautiful theorem of Asplund and Gr\"{u}nbaum \cite{AG}, which states that every intersection graph $G$ of axis-parallel rectangles in the plane satisfies $\chi(G)\le 4(\omega(G))^2$.

A family $\mathcal{G}$ of graphs is \emph{$\chi$-bounded} if there exists a function $f:\mathbb{Z}^{+}\rightarrow \mathbb{Z}^{+}$ such that every $G\in\mathcal{G}$ satisfies $\chi(G)\leq f(\omega(G))$. In this case, say that the function $f$ is \emph{$\chi$-bounding for $\mathcal{G}$}. Using this terminology, the result of Asplund and Gr\"{u}nbaum \cite{AG} mentioned above can be rephrased as follows: The family of intersection graphs of axis-parallel rectangles in the plane is $\chi$-bounded with bounding function $f(k)=4k^2$. (It is conjectured that the same is true with bounding function $f(k)=O(k)$.) However, an ingenious construction of Burling \cite{B} shows that the family of intersection graphs of axis-parallel boxes in $\mathbb{R}^{3}$ is {\em not} $\chi$-bounded. The $\chi$-boundedness of intersection graphs of chords of a circle was established by Gy\'arf\'as~\cite{Gy, Gy1}; see also Kostochka {\em et al.}~\cite{Ko1, KK}. Deciding whether a family of graphs is $\chi$-bounded is often a difficult task \cite{Ko2}.

Computing the chromatic number of the {\em disjointness graph} of a family of objects, $\mathcal{C}$, is equivalent to determining the {\em clique cover number} of the corresponding intersection graph $G$, that is, the minimum number of cliques whose vertices together cover the vertex set of $G$. This problem can be solved in polynomial time only for some very special families (for instance, if $\mathcal{C}$ consists of intervals along a line or arcs along a circle~\cite{Gav}). On the other hand, the problem is known to be NP-complete if $\mathcal{C}$ is a family of chords of a circle~\cite{KeS, GaJ} or a family of unit disks in the plane~\cite{Sup, Cer}, and in many other cases. There is a vast literature providing approximation algorithms or inapproximability results for the clique cover number~\cite{Dum, Eid}.

\paragraph{Families of curves.} A \emph{curve} or {\em string} in $\mathbb{R}^{2}$ is the image of a continuous function $\phi:[0,1]\rightarrow \mathbb{R}^{d}$. A curve $C\subset\mathbb{R}^{2}$ is called \emph{$x$-monotone} if every vertical line intersects $C$ in at most one point. We say that $C$  is \emph{grounded at the curve $L$} if one of the endpoints of $C$ is in $L$, and this is the only intersection point of $C$ and $L$. A {\em grounded $x$-monotone curve} is an $x$-monotone curve that is contained in the half-plane $\{x\geq 0\}$, and whose left endpoint lies on the vertical line $\{x=0\}$.

It was first suggested by Erd\H os in the 1970s, and remained the prevailing conjecture for 40 years, that the family of intersection graphs of curves (the family of so-called ``string graphs'') is $\chi$-bounded~\cite{Bra, KoN}. There were many promising facts pointing in this direction. Extending earlier results of McGuinness~\cite{M}, Suk~\cite{Suk}, and Laso\'n~{\em et al.}~\cite{Las}, Rok and Walczak~\cite{Rok1, Rok2} proved the conjecture for {\em grounded} families of curves. Nevertheless, in 2014, Pawlik {\em et al.}~\cite{PKKLMTW} disproved Erd\H os's conjecture. They managed to modify Burling's above mentioned construction to obtain a sequence of finite families of {\em segments} in the plane whose intersection graphs, $G_n$, are triangle-free (that is, $\omega(G_n)=2$), but their chromatic numbers tend to infinity, as $n\rightarrow\infty$.

Recently, Pach, Tardos and T\'{o}th~\cite{PTT} proved that the family of {\em disjointness graphs} of curves in the plane is not $\chi$-bounded either; see also~\cite{MWW}. However, the situation is different if we restrict our attention to {\em $x$-monotone} curves. It was shown in~\cite{PaT, Lar} that the family of {disjointness graphs} of $x$-monotone curves in the plane is $\chi$-bounded with a bounding function $f(k)=k^4$. For {\em grounded $x$-monotone} curves, the same proof provides a better bounding function: $f(k)=k^2$. These results proved 25 years ago were not likely to be tight. However, in spite of many efforts, no-one has managed to improve them or to show that they are optimal.

\paragraph{Our results.} The aim of the present paper is to fill this gap. We proved, much to our surprise, that the order of magnitude of the last two bounds cannot be improved. In fact, in the case of grounded $x$-monotone curves, we determined the exact value of the best bounding function for every $k\ge 2$. To the best of our knowledge, this is the first large family of non-perfect geometric disjointness graphs, for which one can precisely determine the best bounding function.

\begin{theorem}\label{upperbound}
	Let $G$ be the disjointness graph of a family of grounded $x$-monotone curves. If ${\omega(G)=k}$, then $\chi(G)\leq \binom{k+1}{2}$.
\end{theorem}	

\begin{theorem}\label{construction}
	For every positive integer $k\geq 2$, there exists a family $\mathcal{C}$ of grounded $x$-monotone curves such that if $G$ is the disjointness graph of $\mathcal{C}$, then $\omega(G)=k$ and $\chi(G)=\binom{k+1}{2}$.
\end{theorem}

It turns out that disjointness graphs of grounded $x$-monotone curves can be characterized by two total orders defined on their vertex sets that satisfy some special properties. This observation is the key idea behind the proof of the above two theorems.

The disjointness graph of any collection of $x$-monotone curves, each of which intersects a given vertical line (the $y$-axis, say), is the intersection of two disjointness graphs of grounded $x$-monotone curves. The methods used for proving Theorems~\ref{upperbound} and~\ref{construction} can be extended to such disjointness graphs and yield sharp bounds.

\begin{theorem}\label{upperbound2}
	Let $G$ be the disjointness graph of a family $\mathcal{C}$ of $x$-monotone curves such that all elements of $\mathcal{C}$ have nonempty intersection with a vertical line $l$. If ${\omega(G)=k}$, then $\chi(G)\leq \frac{k+1}{2}\binom{k+2}{3}$.
\end{theorem}	

\begin{theorem}\label{construction2}
	For every positive integer $k\geq 2$, there exists a family $\mathcal{C}$ of $x$-monotone curves such that all elements of $\mathcal{C}$ have nonempty intersection with a vertical line $l$, the disjointness graph $G$ of $\mathcal{C}$ satisfies $\omega(G)=k$, and $\chi(G)=\frac{k+1}{2}\binom{k+2}{3}$.
\end{theorem}

As we have mentioned before, according to \cite{PaT, Lar}, $k^4$ is a bounding function for disjointness graphs of {\em any} family of $x$-monotone curves. Theorem~\ref{construction2} implies that the order of magnitude of this bounding function is best possible. Actually, we can obtain a little more.

\begin{theorem}\label{cor}
	For any positive integer $k$, let $f(k)$ denote the smallest $m$ such that any $K_{k+1}$-free disjointness graph of $x$-monotone curves can be properly colored with $m$ colors. Then we have
	
	$$\frac{k+1}{2}\binom{k+2}{3}\leq f(k)\leq k^{2}\binom{k+1}{2}.$$
\end{theorem}

Here the lower bound follows directly from Theorem~\ref{construction2}.

Our paper is organized as follows. In Section \ref{sect:grounded}, we prove Theorem~\ref{upperbound} and the upper bound in Theorem~\ref{cor}. The existence of the graphs satisfying Theorem~\ref{construction} is proved in Section~\ref{grounded:lower}, using probabilistic techniques. The proofs of Theorems~\ref{upperbound2} and \ref{construction2} are presented in Sections~\ref{sect:vertical} and \ref{vertical:lower}, respectively. The last section contains open problems and concluding remarks.

\section{A bounding function for grounded curves\\--Proofs of Theorems~\ref{upperbound} and~\ref{cor}}\label{sect:grounded}

First, we establish Theorem \ref{upperbound}. As usual, we denote the set $\{1,2,\dots,n\}$ by $[n]$.

An {\em ordered graph} $G_{<}$ is a graph, whose vertex set is endowed with a total ordering $<$. Ordered graphs are often more suitable for modelling geometric configurations than unordered ones; see, e.g., \cite{F,PT}. To model families of grounded $x$-monotone curves, we introduce a class of ordered graphs.

\begin{definition}
An {\em ordered graph} $G_{<}$ is called a \emph{semi-comparability graph}, if it has no 4 vertices $a,b,c,d\in V(G_{<})$ such that $a<b<c<d$ and $ab,bc,cd\in E(G_{<})$, but $ac,bd\not\in E(G_{<})$.

An {\em unordered graph} $G$ is said to be a {\em semi-comparability graph}, if its vertex set has a total ordering $<$ such that $G_{<}$ is a semi-comparability graph.
\end{definition}

Obviously, every \emph{comparability graph} (that is, every graph whose edge set consists of all comparable pairs of a partially ordered set) is a semi-comparability graph.

\begin{lemma}\label{semiposet0}
The disjointness graph of every family $\mathcal{C}$ of grounded $x$-monotone curves is a semi-comparability graph.
\end{lemma}

\begin{proof}
Let $G$ be the disjointness graph of $\mathcal{C}$. Identify the vertices of $G$ with the elements of $\mathcal{C}$. For any $\gamma\in\mathcal{C}$, let $(0,y_{\gamma})$ be the left endpoint of $\gamma$. Slightly perturbing the curves if necessary, we can assume without loss of generality that no two $y_{\gamma}$s coincide. Let $<$ be the total ordering on $V(G)$, according to which $\gamma<\gamma'$ if and only if $y_{\gamma}<y_{\gamma'}$.

Suppose for contradiction that there exist 4 curves $a,b,c,d$ such that $a<b<c<d$ and $ab,bc,cd\in E(G)$, but $ac,bd\not\in E(G)$. Then $a$ and $c$ must intersect, which means that $a$, $c$, and the ground line $x=0$ enclose a region $A$. Since $b$ does not intersect either of $a$ or $c$, it must lie in $A$. In order to intersect $b$, $d$ has to cross $c$, which is a contradiction.
\end{proof}

By Dilworth's theorem \cite{D}, comparability graphs are perfect. Thus, any comparability graph $G$ can be properly colored with $\omega(G)$ colors. While not all semi-comparability graphs are perfect, they are $\chi$-bounded.

\begin{lemma}\label{semiposet}
	For any semi-comparability graph $G$ with $\omega(G)=k$, we have $\chi(G)\leq \binom{k+1}{2}$.
\end{lemma}

\begin{proof}
Fix an ordering $<$ of $V(G)$ such that $G_{<}$ is a semi-comparability graph. For every $v\in V(G)$, let $f(v)$ denote the size of the largest clique with minimal element $v$. Then $f(v)\in [k]$. For $i=1,\dots,k$, let $V_{i}=\{v\in G:f(v)=i\}$.
	
	The main observation is that $G[V_{i}]$ is a partial order. Indeed, suppose to the contrary that there exist 3 vertices $a,b,c\in V_{i}$ such that $a<b<c$ and $ab,bc\in E(G)$, but $ac\not\in E(G)$. Let $C\subset V(G)$ be a clique of size $i$ with minimal element $c$. If $d\in C\setminus \{c\}$, then $b$ and $d$ must be joined by an edge, otherwise the quadruple $a,b,c,d$ satisfies the conditions $ab,bc,cd\in E(G)$ and $ac,bd\not\in E(G)$. Thus, $b$ is joined to every vertex in $C$ by an edge, which means that $C\cup\{b\}$ is a clique of size $i+1$ with minimal element $b$, contradicting our assumption that $b\in V_{i}$.
	
	Hence, every $G[V_{i}]$ is a partial order. Using the fact that $G[V_{i}]$ does not contain a clique of size $i+1$, by Dilworth's theorem \cite{D} we obtain that $\chi(G[V_{i}])\leq i$. Summing up for all $i$, we get that $\chi(G)\leq \sum_{i=1}^{k}\chi(G[V_{i}])\leq \binom{k+1}{2},$ as required.
\end{proof}

The combination of Lemmas \ref{semiposet0} and \ref{semiposet} immediately implies Theorem \ref{upperbound}.

Next, we prove the upper bound in Theorem~\ref{cor}.

\begin{theorem}\label{cor2}
	Let $G$ be the disjointness graph of a collection of $x$-monotone curves with $\omega(G)=k$. Then we have $\chi(G)\leq k^{3}(k+1)/2$.
\end{theorem}

\begin{proof}
	Let $\mathcal C$ be a collection of $x$-monotone curves satisfying the conditions in the theorem. For any $\gamma\in\mathcal{C}$, let $x(\gamma)$ denote the projection of $\gamma$ to the $x$-axis. For $\alpha,\beta\in \mathcal{C}$, let $\alpha\prec\beta$ if $\min x(\alpha)<\min x(\beta)$ and $\max x(\alpha)<\max x(\beta)$.

Suppose that $\alpha$ and $\beta$ are disjoint. Let $\alpha<_{1}\beta$ if $\alpha\prec \beta$ and $\alpha$ is {\em below} $\beta$, that is, if on every vertical line that intersects both $\alpha$ and $\beta$, the intersection point of $\alpha$ lies below the intersection point of $\beta$. Let $\alpha<_{2}\beta$ if $\alpha\prec \beta$ and $\beta$ is below $\alpha$. Clearly, $<_{1}$ are $<_{2}$ are partial orders.
	
	As $\omega(G)\leq k$, the size of the longest chains with respect to $<_{1}$ and $<_{2}$ is at most $k$. Therefore, the vertices of $G$ can be colored with $k^{2}$ colors such that each color class is an antichain in both $<_{1}$ and $<_{2}$.
	
	It remains to show that each of these color classes can be properly colored with $\binom{k+1}{2}$ colors. Let $\mathcal{C}'\subset\mathcal{C}$ such that no two elements of $\mathcal{C}'$ are comparable by $<_{1}$ or $<_{2}$. Then, if $\alpha,\beta\in \mathcal{C}'$, then either $\alpha$ and $\beta$ intersect, or one of the intervals $x(\alpha)$ or $x(\beta)$ contains the other. In either case, $x(\alpha)$ and $x(\beta)$ have a nonempty intersection, so any two elements of $\{x(\gamma):\gamma\in\mathcal{C}'\}$ intersect. Hence, $\bigcap_{\gamma\in\mathcal{C}'}x(\gamma)$ is nonempty, and there exists a vertical line $l$ that intersects every element of $\mathcal{C}'$.
	
	Let $G'$ denote the disjointness graph of $\mathcal{C}'$. Order the elements of $\mathcal{C}'$ with respect to their intersections with $l$, from bottom to top. We claim that the resulting ordered graph $G'_{<}$ is a {\em semi-comparability graph}. Indeed, suppose to the contrary that there are four vertices $a,b,c,d\in V(G')$ such that $a<b<c<d$ and  $ab,bc,cd\in E(G')$, but $ac,bd\not\in E(G')$. Without loss of generality, suppose that the length of $x(b)$ is larger than the length of $x(c)$; the other case can be handled similarly. As $bc\in E(G')$, we have $x(c)\subset x(b)$ and $b$ is below $c$, so every vertical line intersecting $c$ intersects $b$ as well, and its intersection with $b$ lies below its intersection with $c$. Also, as $ab\in E(G')$, we have that $a$ is below $b$. But then $a$ and $c$ must be disjoint, contradicting the condition $ac\not\in E(G')$.
	
	Thus, we can apply Lemma \ref{semiposet} to conclude that $G'$ can be properly colored with $\binom{k+1}{2}$ colors. This completes the proof.
\end{proof}

Let $g(n)$ denote the maximal number $m$ such that every collection of $n$ convex sets in the plane contains $m$ elements that are either pairwise disjoint, or pairwise intersecting. Larman et al.~\cite{LMPT} proved that $g(n)\geq n^{1/5}$, while the best known upper bound, due to Kyn\v{c}l \cite{Ky} is $g(n)<n^{\log 8/\log 169}\approx n^{0.405}$. Theorem~\ref{cor2} implies the following modest improvement on the lower bound.

\begin{corollary}\label{corcor}
	Every collection of $n$ $x$-monotone curves (or convex sets) in the plane contains $((2+o(1))n)^{1/5}\approx 1.15n^{1/5}$ elements that are either pairwise disjoint or pairwise intersecting.
\end{corollary}

\begin{proof}
	In every graph $G$ on $n$ vertices, we have $\alpha(G)\chi(G)\geq n$. In view of Theorem~\ref{cor2}, this implies that if $\mathcal{C}$ is a collection of $n$ $x$-monotone curves and $G$ is the disjointness graph of $\mathcal{C}$, then we have $$\alpha(G)(\omega(G))^{3}\frac{\omega(G)+1}{2}\geq n.$$
Therefore, $\max\{\alpha(G),\omega(G)\}\geq ((2+o(1))n)^{1/5}$, as claimed.
\end{proof}

\section{Magical graphs--Proof of Theorem~\ref{construction}}\label{grounded:lower}
The converse of Lemma \ref{semiposet0} is not true: not every semi-comparability graph can be realized as the disjointness graph of a collection of grounded $x$-monotone curves. See Section \ref{sect:remarks}, for further discussion. To characterize such disjointness graphs, we need to introduce a new family of graphs.

A graph $G_{<_{1},<_{2}}$  with two total orderings, $<_{1}$ and $<_{2}$, on its vertex set is called \emph{double-ordered}. If the orderings $<_{1},<_{2}$ are clear from the context, we shall write $G$ instead of $G_{<_{1},<_{2}}$.

\begin{definition}
A double-ordered graph $G_{<_{1},<_{2}}$ is called \emph{magical} if for any three distinct vertices $a,b,c\in V(G)$ with $a<_{1}b<_{1}c$, if $ab,bc\in E(G)$ and $ac\not\in E(G)$, then $b<_{2}a$ and $b<_{2}c$.

A graph $G$ is said to be {\em magical}, if there exist two total orders $<_{1},<_{2}$ on $V(G)$ such that $G_{<_{1},<_{2}}$ is magical. In this case, we say that the pair $(<_{1},<_{2})$ \emph{witnesses} $G$.
\end{definition}

It easily follows from the above definition that if $G_{<_{1},<_{2}}$ is magical, then $G_{<_{1}}$ is a semi-comparability graph.

\begin{lemma}\label{magical1}
	If $\mathcal{C}$ is a collection of grounded $x$-monotone curves, then the disjointness graph of $\mathcal{C}$ is magical.
\end{lemma}

\begin{proof}
	Let $G$ be the disjointness graph of $\mathcal{C}$, and identify the vertices of $G$ with the elements of $\mathcal{C}$. For any $\gamma\in\mathcal{C}$, let $(0,y_{\gamma})$ be the endpoint of $\gamma$ lying on the vertical axis $\{x=0\}$, and let $(x_{\gamma},y'_{\gamma})$ be the other endpoint of $\gamma$.
	
	 Define the total orderings $<_{1}$ and $<_{2}$ on $V(G)$, as follows.  Let $\gamma<_{1}\gamma'$ if and only if $y_{\gamma}<y_{\gamma'}$, and let $\gamma<_{2}\gamma'$ if and only if $x_{\gamma}<x_{\gamma'}$.
	
	 Suppose that for a triple $a,b,c\in \mathcal{C}$ we have that $a<_{1}b<_{1}c$ and $ab,bc\in E(G)$, but $ac\not\in E(G)$. Then $a$ and $c$ intersect. Hence, $a,c,$ and the ground curve $\{x=0\}$ enclose a region $A$, and $b\subset A$. This implies that the $x$-coordinate of the right endpoint of $b$ is smaller than the $x$-coordinates of the right endpoints of $a$ and $c$. Therefore, we have $b<_{2}a$ and $b<_{2}c$, showing that $G$ is magical.
\end{proof}

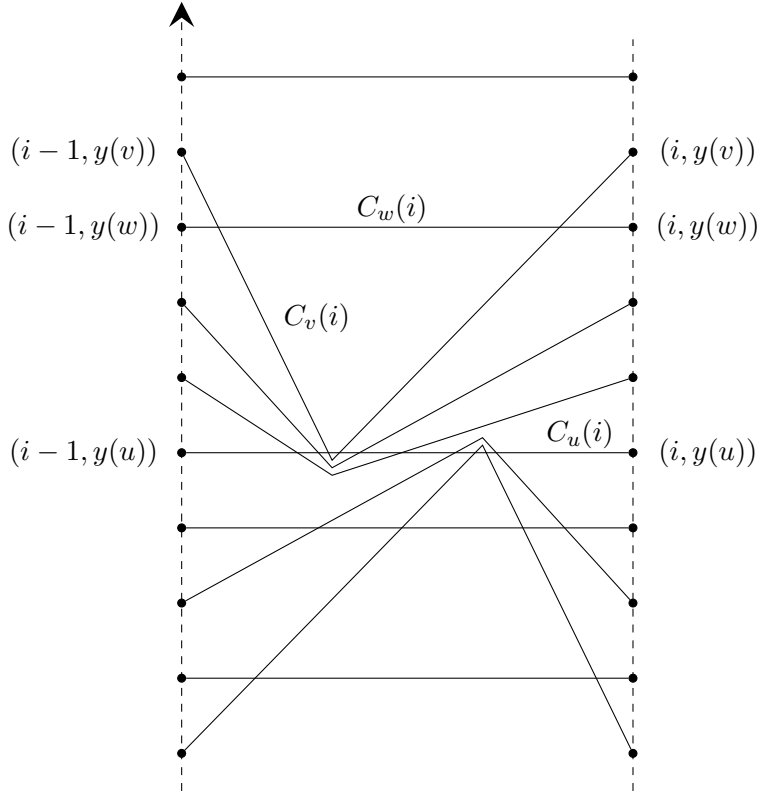
\begin{figure}[t]
	\begin{center}
		\begin{tikzpicture}[scale=1]
		
	%	\draw[->,dashed] (0,0) -- (0,10) ;
		\draw [dashed,decoration={markings,mark=at position 1 with
			{\arrow[scale=3,>=stealth]{>}}},postaction={decorate}] (0,0) -- (0,10.5);
		
		\draw[dashed] (6,0) -- (6,10) ;
		\foreach \i in {0,...,9}
		{
			\node[vertex] (A\i) at (0,\i+0.5) {};
			\node[vertex] (B\i) at (6,\i+0.5) {};
		}
		\node[] at (-1.3,8.5) {$(i-1,y(v))$} ;
		\node[] at (7,8.5) {$(i,y(v))$} ;
		\node[] at (-1.3,7.5) {$(i-1,y(w))$} ;
		\node[] at (7,7.5) {$(i,y(w))$} ;
		\node[] at (2.8,7.75) {$C_{w}(i)$} ;
		\node[] at (1.8,6.3) {$C_{v}(i)$} ;
        \node[] at (5.3,4.75) {$C_{u}(i)$} ;
		
		\node[] at (-1.3,4.5) {$(i-1,y(u))$} ;
		\node[] at (7,4.5) {$(i,y(u))$} ;
		
		\draw (A0) -- (4,4.6) -- (B0) ;
		\draw (A1) -- (B1) ; 	
		\draw (A2) -- (4,4.7) -- (B2) ;
		\draw (A3) -- (B3) ;
		\draw (A4) -- (B4) ;
		
		\draw (A5) -- (2,4.2) -- (B5) ;
		\draw (A6) -- (2,4.3) -- (B6) ;
		\draw (A7) -- (B7) ;
		\draw (A8) -- (2,4.4) -- (B8) ;
		\draw (A9) -- (B9) ;
		
		%		\node[] at (-1,6.75) {$V_{1,1}$};

		\end{tikzpicture}
		\caption{An illustration of the curves $C_{v}(i)$ in the proof of Lemma \ref{magical2}.}
		\label{figure1}
	\end{center}
\end{figure}

\begin{lemma}\label{magical2}
	Let $G$ be a magical graph. Then there exists a family $\mathcal{C}$ of grounded $x$-monotone curves such that the disjointness graph of $\mathcal{C}$ is isomorphic to $G$.
\end{lemma}

\begin{proof}
	Let $n$ be the number of vertices of $G$. Let $<_{1}$ and $<_{2}$ be total orderings on $V(G)$ witnessing that $G$ is magical. For any vertex $v\in V(G)$, let $y(v)\in [n]$ denote the position of $v$ in the ordering $<_{1}$, and let $x(v)$ denote the position of $v$ in the ordering $<_{2}$.
	
	For any $v\in V(G)$, we define an $x$-monotone curve $C_{v}$, which will be composed of $x(v)$ smaller $x$-monotone pieces, $C_{v}(1),\dots, C_{v}(x(v)),$ such that $C_{v}(i)$ starts at the point $(i-1,y(v))$, and ends at the point $(i,y(v))$. The pieces $C_{v}(i)$ are defined, as follows.
	
	Let $u\in V(G)$ such that $x(u)=i$. If $u=v$ or there is an edge between $u$ and $v$, then let $C_{v}(i)$ be the horizontal line segment connecting $(i-1,y(v))$ and $(i,y(v))$. Otherwise, let $C_{v}(i)$ be a polygonal curve consisting of two segments, whose 3 vertices are
$$\left(i-1,y(v)\right),\;\;\left(i-\frac23,y(u)-\frac{1}{10}+\frac{y(v)}{10n}\right),\;\;\left(i,y(v)\right)\;\; \mbox{  if  }\;\; y(u)<y(v), $$
or
$$\left(i-1,y(v)\right),\;\;\;\;\left(i-\frac13,y(u)+\frac{y(v)}{10n}\right),\;\;\;\;\left(i,y(v)\right)\;\;\;\;\;\;\;\; \mbox{  if  }\;\; y(u)>y(v).$$
See Figure \ref{figure1} for an illustration. One can easily check the following property of the curves $\{C_{v}(i)\}_{v\in V(G)}$.  If $v,w\in V(G)$ are distinct vertices such that $C_{v}(i)$ and $C_{w}(i)$ intersect, then
	
	(i) $x(v),x(w)\geq i$.
	
	(ii) Exactly one of $v$ and $w$ is joined to $u$ in $G$. Without loss of generality, suppose that it is $w$.
	
	(iii) Then $y(u)\le y(w) < y(v)$ or $y(v)<y(w)\leq y(u)$.
	
\noindent Now we show that $G$ is the disjointness graph of $\mathcal{C}=\{C_{v}:v\in V(G)\}$.
	
	If $v$ and $w$ are not joined by an edge in $G$, then $C_{v}(\min\{x(v),x(w)\})$ and $C_{w}(\min\{x(v),x(w)\})$ intersect by definition, so $C_{v}$ and $C_{w}$ have a nonempty intersection.

Our task is reduced to showing that if $v$ and $w$ are joined by an edge, then $C_{v}$ and $C_{w}$ do not intersect. Suppose to the contrary that $C_{v}$ and $C_{w}$ intersect. Then there exists ${i\in [\min\{x(v),x(w)\}-1]}$ such that $C_{v}(i)$ and $C_{w}(i)$ intersect. Let $u$ be the vertex satisfying $x(u)=i$. Then either $y(u)\le y(v),y(w)$, or $y(u)\le y(v),y(w)$. Without loss of generality, let $y(u)\le y(v),y(w)$, the other case can be handled in a similar manner. Again, without loss of generality, we can suppose that  $y(w)<y(v)$. Then $C_{v}(i)$ intersects $C_{u}(i)$, and $C_{w}(i)$ is disjoint from $C_{u}(i)$, or equivalently, $uw\in E(G)$, but $uv\not\in E(G)$. However, this is impossible, because $wv\in E(G)$, so the triple $u,w,v$ would contradict the assumption that $G$ is magical.
\end{proof}

By Lemma~\ref{magical2}, in order to prove Theorem~\ref{construction}, it is enough to verify the corresponding statement for magical graphs. In other words, we have to prove the following.

\begin{theorem}\label{mainmagical}
	For every positive integer $k\geq 2$, there exists a magical graph $G$ such that $\omega(G)=k$ and $\chi(G)=\binom{k+1}{2}$.
\end{theorem}

The rest of this section is devoted to the proof of this theorem. The proof is probabilistic and is inspired by a construction of Kor\'andi and Tomon \cite{KT}. We shall consider a random double-ordered graph with certain parameters, and show that the smallest magical graph covering its edges meets the requirements in Theorem \ref{mainmagical}. To accomplish this plan, we first examine how the smallest magical graph covering the edges of a given double-ordered graph looks like.

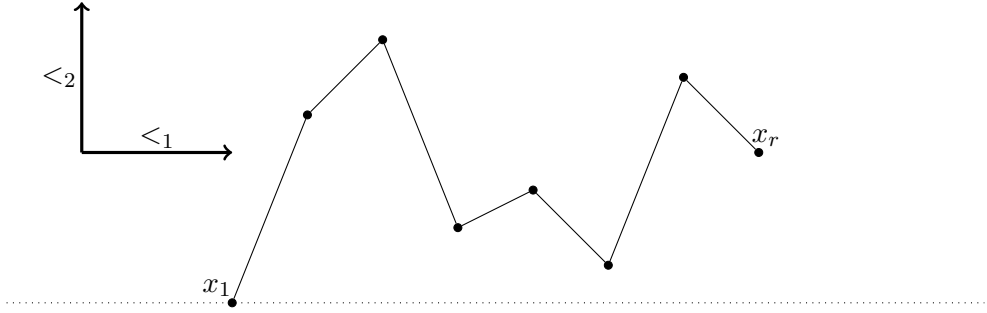
\begin{figure}[t]
	\begin{center}
		\begin{tikzpicture}[scale=1]
		
	    \node[vertex] (A1) at (0,0) {} ;
	    \node[vertex] (A2) at (1,2.5) {} ;
	    \node[vertex] (A3) at (2,3.5) {} ;
	    \node[vertex] (A4) at (3,1) {} ;
	    \node[vertex] (A5) at (4,1.5) {} ;
	    \node[vertex] (A6) at (5,0.5) {} ;
	    \node[vertex] (A7) at (6,3) {} ;
	    \node[vertex] (A8) at (7,2) {} ;
	
	    \draw (A1) -- (A2) -- (A3) -- (A4) -- (A5) -- (A6) -- (A7) -- (A8) ;
	    \draw[dotted] (-3,0) -- (10,0) ;
	
	    \node[] at (-0.2,0.2) {$x_{1}$} ;
	    \node[] at (7.1,2.2) {$x_{r}$} ;
	
	    \draw[->, very thick] (-2,2) -- (0,2) ;
	    \draw[->, very thick] (-2,2) -- (-2,4) ;
	    \node[] at (-1,2.2) {$<_{1}$} ;
	    \node[] at (-2.3,3) {$<_{2}$} ;

		\end{tikzpicture}
		\caption{A mountain path. The dotted line shows the minimum of $x_{1}$ and $x_{r}$ in $<_{2}$, so all the other points of the path must be above it.}
		\label{figure4}
	\end{center}
\end{figure}

Let $G_{<_{1},<_{2}}$ be a double-ordered graph. A sequence of vertices $x_{1},\dots,x_{r}\in V(G)$ is said to form a \emph{mountain-path}, if $x_{1}<_{1}...<_{1} x_{r}$, $x_{i}x_{i+1}\in E(G)$ for every $i (1\le i<r)$, and either $x_{1}<_{2}x_{2},...,x_{r-1}$ or $x_{r}<_{2}x_{2},\dots,x_{r-1}$. See Figure \ref{figure4}.

\begin{lemma}\label{mountainpaths}
	Let $G_{<_{1},<_{2}}$ be a double-ordered graph. There exists a unique minimal graph $G'_{<_{1},<_{2}}$ on $V(G)$ such that $E(G)\subset E(G')$ and $G'_{<_{1},<_{2}}$ is magical. Moreover, if $u,v\in V(G)$, then $u$ and $v$ are joined by an edge in $G'$ if and only if there exists a mountain-path connecting $u$ and $v$.
\end{lemma}

\begin{proof}
	Let $H=H_{<_{1},<_{2}}$ be any magical graph on the vertex set $V(G)$ such that $E(G)\subseteq E(H)$. Let $x_{1},...,x_{r}$ be a mountain-path in $G$ with $x_{1}<_{2}x_{r}$. Using the definition of magical graphs, it is easy to prove by induction on $i$ that $x_{1}$ and $x_{i}$ are joined by an edge in $E(H)$, for every $i>1$. Therefore, we have $x_{1}x_{r}\in E(H)$. (We can proceed similarly if $x_{r}<_{2}x_{1}$.)
	
	With a slight abuse of notation, from now on let $H=H_{<_{1},<_{2}}$ denote the double-ordered graph on $V(G)$, in which $u$ and $v$ are joined by an edge if and only if there exists a mountain-path connecting $u$ to $v$. We will show that $H$ is magical, that is, for every triple $u,v,w\in V(G)$, the following holds: if $u<_{1}v<_{1}w$ such that  $uv, vw\in E(H)$, and $u<_{2}v$ or $w<_{2}v$, then $uw\in E(H)$. As $uv,vw\in E(H)$, there exist two  mountain-paths $u=x_{1},x_2,...,x_{r}=v$ and $v=x_{1}',x_2',...,x'_{r'}=w$. However, this implies that $u=x_{1},\dots,x_{r},x'_{2},\dots,x'_{r'}=w$ is a mountain-path between $u$ and $w$, so that $uw\in E(H)$.
\end{proof}

For the rest of the discussion, we need to introduce a few parameters that depend on $k$. Set $\lambda=1/k^{2}$, $t=20k^{2}\log k$,  $h=t^{k^{2}}k^{2k^{2}+8}$, $n=6h$ and $p=t/n$.

Let $S=\{(a,b)\in [k]^{2}: a+b\geq k+1\}$. For each $(a,b)\in S$, let $A_{a,b}$ be a set of $n$ arbitrary points in the interior of the unit square $[ak+b,ak+b+1]\times [bk+a,bk+a+1]$  with distinct $x$ and $y$ coordinates, see Figure \ref{figure2}. Let $V=\bigcup_{(a,b)\in S}A_{a,b}$, and let $<_{1}$ and $<_{2}$ be the total orderings on $V$ induced by the $x$ and $y$ coordinates of the elements of $V$, respectively. A pair of vertices $\{u,v\}$ in $V$ is called \emph{available} if $u\in A_{a,b}, v\in A_{a',b'}$ with $(a,b)\neq (a',b')$.

Let $G_{0}$ denote the {\em random graph} on $V$ in which every available pair of vertices is connected by an edge with probability $p$, independently from each other. $G_{0}$ does not have any edge whose endpoints belong to the same set $A_{a,b}$. Let $G'_{<_{1},<_{2}}$ be the minimal magical graph on $V$ containing all edges of $G_{0}$.

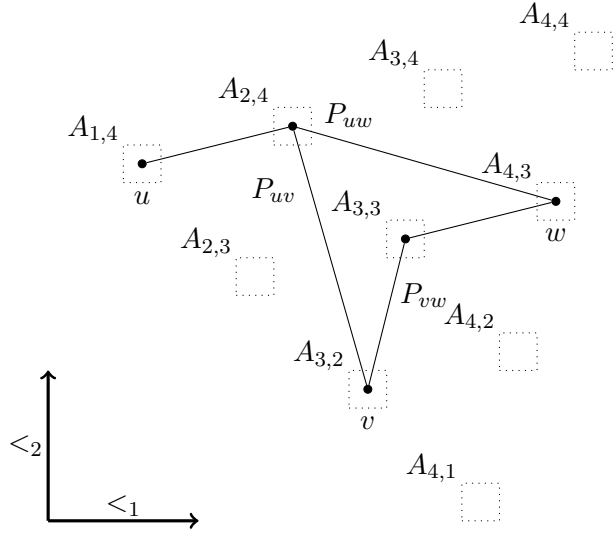
\begin{figure}[t]
	\begin{center}
		\begin{tikzpicture}[scale=2]
		
		\draw[->, very thick] (1.5,2) -- (2.5,2) ;
		\draw[->, very thick] (1.5,2) -- (1.5,3) ;
		\node[] at (2,2.1) {$<_{1}$} ;
		\node[] at (1.35,2.5) {$<_{2}$} ;
		
		\foreach \i in {1,...,4}
		{
			\foreach \j in {1,...,4}
			{
				
				\pgfmathtruncatemacro\result{\i+\j}
				\ifthenelse{\result>4}{\draw[dotted] (\i+\j/4,\j+\i/4) -- (\i+\j/4+0.25,\j+\i/4) -- (\i+\j/4+0.25,\j+\i/4+0.25) -- (\i+\j/4,\j+\i/4+0.25) -- (\i+\j/4,\j+\i/4) ;}{;}			
				\ifthenelse{\result>4}{\node[] at (\i+\j/4-0.2,\j+\i/4+0.35) {$A_{\i,\j}$} ;}{;}
			}
			
		}
		
		\node[vertex] (U) at (2.125,4.375) {} ; \node[] at (2.125,4.15) {$u$} ;
		\node[vertex] (V) at (3.625,2.875) {} ; \node[] at (3.625,2.65) {$v$} ;
		\node[vertex] (W) at (4.875,4.125) {} ; \node[] at (4.875,3.9) {$w$} ;
		\node[vertex] (X1) at (3.125,4.625) {} ;
		\node[vertex] (X2) at (3.875,3.875) {} ;
		
		\draw (U) -- (X1) -- (W) ; \node[] at (3,4.2) {$P_{uv}$} ;
		\draw (X1) -- (V) ; \node[] at (3.5,4.7) {$P_{uw}$} ;
		\draw (V) -- (X2) -- (W) ; \node[] at (4,3.5) {$P_{vw}$} ;
		
		%		\node[] at (-1,6.75) {$V_{1,1}$};

		\end{tikzpicture}
		\caption{An illustration of the sets $A_{a,b}$ for $k=4$, and a hole $(u,v,w)$ which induces a triangle in $G'$.}
		\label{figure2}
	\end{center}
\end{figure}

\begin{claim}\label{independent}
	 With probability at least $2/3$, $G'$ has no independent set larger than $(1+\lambda)n$.
\end{claim}
\begin{proof}
	As $G_{0}$ is a subgraph of $G'$, it is enough to show that $G_{0}$ has no independent set of size larger than $(1+\lambda)n$, with probability at least $2/3$.
	
	 Let $I\subset V$ such that $|I|>(1+\lambda)n$. Then there are at least $\lambda n^{2}/2$ available pairs of vertices, whose both endpoints belong to $I$. Indeed, if $u\in A_{a,b}$, then $\{u,v\}$ is available for every $v\in (I\setminus A_{a,b})$, so there are at least $|I\setminus A_{a,b}|\geq \lambda n$ available pairs containing $u$. Hence, the total number of available pairs in $I$ is at least $|I|\lambda n/2>\lambda n^{2}/2$.
	
Thus, the probability that $I$ is an independent set in $G_{0}$ is at most
	$$(1-p)^{\lambda n^{2}/2}<e^{-p\lambda n^{2}/2}=e^{-t\lambda n/2}.$$
As the number of $(1+\lambda)n$-sized subsets of $V$ is
	$$\binom{|V|}{(1+\lambda)n}<\left(\frac{e|V|}{(1+\lambda n)}\right)^{(1+\lambda)n}<(ek^{2})^{(1+\lambda)n},$$
the probability that there is a $(1+\lambda)n$-sized independent set is less than
	$$(ek^{2})^{(1+\lambda)n}e^{-t\lambda n/2}=e^{(1+2\log k)(1+\lambda)n-t\lambda n/2}<1/3.$$
\end{proof}

A triple $(u,v,w)\in V^{3}$ is said to form a \emph{hole}, if $u<_{1}v<_{1}w$ and $v<_{2}u,w$. Recall that $h=t^{k^{2}}k^{2k^{2}+8}$.

\begin{claim}\label{holes}
	Let $N$ be the number of holes in $V$ that induce a triangle in $G'$. Then $\mathbb{E}(N)<h$.
\end{claim}
\begin{proof}
	Let $(u,v,w)$ be a hole, and let us bound the probability that $u,v,w$ induce a triangle in $G'$. Suppose that $u\in A_{a_{1},b_{1}}$, $v\in A_{a_{2},b_{2}}$ and $w\in A_{a_{3},b_{3}}$. We can assume that the pairs $(a_{1},b_{1}),(a_{2},b_{2}),(a_{3},b_{3})$ are distinct, otherwise $u,v,w$ cannot induce a triangle. %As $(u,v,w)$ is a hole, we must have $a_{1}\leq a_{2}<a_{3}$,  $b_{2}<b_{1}$, and $b_{2}\leq b_{3}$.
	
	If $uv,vw,uw\in E(G')$, then there exist three mountain-paths, $P_{u,v}$, $P_{v,w}$ and $P_{u,w}$, with endpoints $\{u,v\}$, $\{v,w\}$ and $\{u,w\}$, respectively. See Figure \ref{figure2}. Note that each of these paths intersects every $A_{a,b}$ in at most one vertex. As $u<_{1}v<_{1}w$, the only vertex in the intersection of  $P_{uv}$ and $P_{vw}$ is $v$. Moreover, $P_{uw}$ cannot contain $v$ as $v<_{2}u$ and $v<_{2}w$.

Consider the graph $P=P_{uv}\cup P_{vw}\cup P_{uw}$. It is a connected graph, but not a tree, because there are two distinct paths between $u$ and $w$: $P_{uv}\cup P_{vw}$ and $P_{uw}$. Hence, we have $|E(P)|\geq |V(P)|$. Let $\mathcal{P}$ denote the set of all such graphs $P$ that appear in $G_{0}$ with positive probability. Then
	\begin{align*}
		\mathbb{P}(\{u,v,w\}\mbox{ induces a triangle in }G')&=\mathbb{P}(P\mbox{ is a subgraph of }G_{0}\mbox{ for some }P\in\mathcal{P})\\
		&\leq \sum_{P\in\mathcal{P}}\mathbb{P}(P\mbox{ is a subgraph of }G_{0}).
	\end{align*}
	
	 For a fixed $P\in\mathcal{P}$, every edge of $P$ is present in $G_{0}$ independently with probability $p$. Hence, the probability that $P$ is a subgraph of $G_{0}$ is $p^{|E(P)|}$, which is at most $p^{|V(P)|}$. The number of graphs in $\mathcal{P}$ with exactly $m$ vertices is at most $\binom{|V|}{m-3}<(k^{2}n)^{m-3}$, as each member of $\mathcal{P}$ contains the vertices $u,v,w$. Finally, every member of $\mathcal{P}$ has at most $|S|\leq k^{2}$ vertices, so we can write
	
	 $$\sum_{P\in\mathcal{P}}\mathbb{P}(P\mbox{ is a subgraph of }G_{0})\leq \sum_{m=3}^{k^{2}}p^{m}(k^{2}n)^{m-3}<t^{k^{2}}k^{2k^{2}+2}n^{-3}.$$
	 Since the number of holes in $V$ is at most $\binom{|V|}{3}<|V|^{3}<k^{6}n^{3}$, we obtain
	 $$\mathbb{E}(N)<t^{k^{2}}k^{2k^{2}+8}=h.$$
\end{proof}

Applying Markov's inequality, the probability that $V$ contains more than $3h$ holes that induce a triangle in $G'$ is at most $1/3$. Hence, there exists a magical graph $G'$ on $V$ such that $G'$ has no independent set of size $(1+\lambda)n$, and $G'$ contains at most $3h$ triangles whose vertices form a hole. By deleting a vertex of each such hole in $G'$, we obtain a magical graph $G$ with at least $|S|n-3h$ vertices, which has no triangle whose vertices form a hole, and no independent set of size $(1+\lambda)n$.

First, we show that $\chi(G)\geq |S|=\binom{k+1}{2}$. Indeed, if $\chi(G)\leq |S|-1$, then $G$ contains an independent set of size
$$\frac{|V(G)|}{|S|-1}\geq \frac{|S|n-3h}{|S|-1}=\left(1+\frac{1}{|S|-1}\right)n-\frac{3h}{|S|-1}>(1+\lambda)n,$$
contradiction.

It remains to prove that $\omega(G)=k$. Clearly, $\omega(G)\ge k$, otherwise, by Lemma~\ref{semiposet}, we would have $\chi(G)\le \binom{k}{2}$, contradicting the last paragraph. Thus, we have to show that $G$ has no clique of size $k+1$. For this, we need the following observation.

\begin{claim}\label{matrix}
	Let $K$ be a subset of $S$ that does not contain three points $(a_{1},b_{1}),(a_{2},b_{2}),(a_{3},b_{3})$ such that $a_{1}<a_{2}\leq a_{3}$ and $b_{2}\leq b_{1}$ and $b_2<b_3$. Then we have $|K|\leq k$.
\end{claim}

\begin{proof}
	We call $(a_{1},b_{1}),(a_{2},b_{2}),(a_{3},b_{3})$ a \emph{bad triple}, if $a_{1}<a_{2}\leq a_{3}$ and $b_{2}\leq b_{1}$ and $b_{2}<b_{3}$.
	
	Let $S=S_k$. We prove the claim by induction on $k$. For $k=1$, the claim is trivial. Suppose that $k\geq 2$ and that the statement has already been verified for $k-1$. We distinguish two cases.
	
	{\em Case 1:} $K$ contains at most $1$ element from the column $\{(k,b):b\in [k]\}$. Let  $$K'=\{(a,b): (a,b+1)\in K \mbox{ and } a<k \}.$$ Then $|K'|\geq |K|-1$ and $K'\subset S_{k-1}$ does not contain a bad triple. Thus, by the induction hypothesis, we have $|K'|\leq k-1$, which implies that $|K|\leq k$.
	
	{\em Case 2:} $K$ contains $2$ distinct elements of the form $(k,b)$ and $(k,b')$, where $b < b'$. Then $K$ cannot contain $(a,k)$ for any $a\in [k-1]$, otherwise $(a,k),(k,b),(k,b')$ would be a bad triple. Thus, $K$ contains at most one element from the row $\{(k,a):a\in [k]\}$ (it might contain $(k,k)$). Let  $$K'=\{(a,b):(a+1,b)\in K \mbox{ and } b\leq k-1\}.$$
Again, $|K'|\geq |K|-1$ and $K'\subset S_{k-1}$ does not contain a bad triple. By the induction hypothesis, we have $|K'|\leq k-1$ and, hence, $|K|\leq k$.
\end{proof}

Now we are in a position to finish the proof of Theorem \ref{mainmagical}. Let $G$ denote the magical graph obtained from $G'$ by deleting a vertex from each of its holes (see right before Claim~\ref{matrix}). Suppose that $C\subset V$ is a clique in $G$. Then $C$ does not contain a hole and it intersects each $A_{a,b}$ in at most one vertex. Let $K=\{(a,b)\subset S:A_{a,b}\cap C\neq \emptyset\}.$ The condition that $C$ does not contain a hole implies that $K$ does not contain three points $(a_{1},b_{1}),(a_{2},b_{2}),(a_{3},b_{3})$ such that $a_{1}<a_{2}\leq a_{3}$ and $b_{2}\leq b_{1}$ and $b_{2} < b_{3}$. Hence, by Claim \ref{matrix}, we have $|C|=|K|\leq k$. This completes the proof of Theorem \ref{mainmagical} and, hence, the proof of Theorem~\ref{construction}.

\section{Bounding function for curves that intersect a vertical line\\--Proof of Theorem~\ref{upperbound2}}\label{sect:vertical}

A \emph{triple-ordered} graph is a graph $G_{<_{1},<_{2},<_{3}}$ with three total orders $<_{1},<_{2},<_{3}$ on its vertex set.

\begin{definition}\label{2magic}
A {\em triple-ordered graph} $G_{<_{1},<_{2},<_{3}}$ is called \emph{double-magical}, if there exist two magical graphs $G^{1}_{<_{1},<_{2}}$ and $G^{2}_{<_{1},<_{3}}$ on $V(G)$  such that  $E(G_{<_{1},<_{2},<_{3}})=E(G^{1}_{<_{1},<_{2}})\cap E(G^{2}_{<_{1},<_{3}})$. An {\em unordered graph} $G$ is said to be {\em double-magical}, if there exist three total orders $<_{1},<_{2},<_{3}$ on $V(G)$ such that the triple-ordered graph $G_{<_{1},<_{2},<_{3}}$ is double-magical. We say that $G$ is {\em witnessed} by $(<_{1},<_{2},<_{3})$.
\end{definition}

By Lemmas~\ref{magical1} and~\ref{magical2}, it is not hard to characterize disjointness graphs of $x$-monotone curves intersected by a vertical line.

\begin{lemma}\label{dmagical}
	Let $\mathcal{C}$ be a collection of $x$-monotone curves such that each member of $\mathcal{C}$ intersects the vertical line $l$. Then the disjointness graph of $\mathcal{C}$ is double-magical.
\end{lemma}

\begin{proof}
	Without loss of generality, let $l=\{x=0\}$. For each $\gamma\in\mathcal{C}$, let $(-x^{-}_{\gamma},y^{-}_{\gamma})$ be the left endpoint of $\gamma$, let $(0,y_{\gamma})$ be the intersection point of $\gamma$ and $l$, and let $(x^{+}_{\gamma},y^{+}_{\gamma})$ be the right endpoint of $\gamma$. Also, let $\gamma^{-}=\gamma\cap \{x\leq 0\}$ and $\gamma^{+}=\gamma\cap \{x\geq 0\}$, and let $\mathcal{C}^{-}=\{\gamma^{-}:\gamma\in\mathcal{C}\}$ and $\mathcal{C}^{+}=\{\gamma^{+}:\gamma\in\mathcal{C}\}$. Then $\mathcal{C}^{+}$ is a collection of grounded curves, and $\mathcal{C}^{-}$ is the reflection of a collection of grounded curves to the line $l$.
	
	  Let $G$, $G^{-}$ and $G^{+}$ be the disjointness graphs of $\mathcal{C}$, $\mathcal{C}^{-}$ and $\mathcal{C}^{+}$, respectively, such that we identify $\gamma,\gamma^{-}$ and $\gamma^{+}$ as the vertices of these graphs for every $\gamma\in\mathcal{C}$. Then $E(G)=E(G^{-})\cap E(G^{+})$.  Let $<_{1}$ be the total ordering on $\mathcal{C}$ defined by $\gamma<_{1}\gamma'$ if $y_{\gamma}<y_{\gamma'}$, let $<_{2}$ be the ordering defined by $\gamma<_{2}\gamma'$ if $x^{-}_{\gamma}<x^{-}_{\gamma'}$, and let $<_{3}$ be the ordering defined by $\gamma<_{3}\gamma'$ if $x^{+}_{\gamma}<x^{+}_{\gamma'}$.
	By Lemma \ref{magical1}, $G^{-}_{<_{1},<_{2}}$ and $G^{+}_{<_{1},<_{3}}$ are magical, so $G_{<_{1},<_{2},<_{3}}$ is double-magical.
\end{proof}

We can just as easily prove the converse of Lemma~\ref{dmagical}, using Lemma \ref{magical2}.

\begin{lemma}\label{dmagical2}
	Let $G$ be a double-magical graph. Then there exists a collection of curves $\mathcal{C}$ such that each member of $\mathcal{C}$ has a nonempty intersection with the vertical line $\{x=0\}$, and the disjointness graph of $\mathcal{C}$ is isomorphic to $G$.
\end{lemma}

\begin{proof}
	Let $(<_{1},<_{2},<_{3})$ be total orders on $V(G)$ witnessing that $G$ is double-magical, and let $G^{1}_{<_{1},<_{2}},G^{2}_{<_{1},<_{3}}$ be two magical graphs on $V(G)$ such that $E(G)=E(G^{1})\cap E(G^{2})$.
	
	Let $|V(G)|=n$. By Lemma \ref{magical2}, there exist $n$ grounded $x$-monotone curves $\gamma_{1}^{+},\dots,\gamma_{n}^{+}$ such that $\gamma_{i}^{+}$ is contained in the nonnegative plane $\{x\geq 0\}$ with one endpoint at $(0,i)$, the disjointness graph of $\{\gamma_{1}^{+},\dots,\gamma_{n}^{+}\}$ is $G^{1}$, and $\gamma_{i}^{+}$ corresponds to the $i$-th vertex of $G^{1}$ in the order $<_{1}$. Also, there exist $n$ $x$-monotone curves $\gamma_{1}^{-},\dots,\gamma_{n}^{-}$ such that $\gamma_{i}^{-}$ is contained in the nonpositive plane $\{x\leq 0\}$ with one endpoint at $(0,i)$, the disjointness graph of $\{\gamma_{1}^{-},\dots,\gamma_{n}^{-}\}$ is $G^{2}$, and $\gamma_{i}^{-}$ corresponds to the $i$-th vertex of $G^{2}$ in the order $<_{1}$. For $i=1,\dots,n$, set $\gamma_{i}=\gamma_{i}^{-}\cup\gamma_{i}^{+}$, then the disjointness graph of $\mathcal{C}=\{\gamma_{i}:i\in [n]\}$ is isomorphic to $G$, and every curve in $G$ has a nonempty intersection with the vertical line $\{x=0\}$.
\end{proof}

For any double-magical graph $G=G_{<_{1},<_{2},<_{3}}$, define four partial orders $\prec_{1},\prec_{2},\prec_{3},\prec_{4}$ on $V(G)$, as follows.
For $a,b\in V(G)$, let

(i) \;\,$a\prec_{1}b$\;\;\;\, if $a<_{1}b$,\, $a<_{2}b$,\, $a<_{3}b$,\;\; and\; $ab\in E(G)$;

(ii) \;$a\prec_{2}b$\;\;\; if $a<_{1}b$,\; $b<_{2}a$,\; $b<_{3}a$,\;\; and\; $ab\in E(G)$;

(iii) $a\prec_{3}b$\;\;\; if $a<_{1}b$,\; $a<_{2}b$,\; $b<_{3}a$,\;\; and\; $ab\in E(G)$;

(iv) $a\prec_{4}b$\;\;\; if $a<_{1}b$,\; $b<_{2}a$,\; $a<_{3}b$,\;\; and\; $ab\in E(G)$.
\smallskip

It follows easily from the definition of double-magical graphs that these are indeed partial orders.
Moreover, they satisfy the following conditions.

(1) If $ab\in E(G)$, then $a$ and $b$ are comparable by precisely one of these $4$ partial orders.

(2) For any $a,b,c\in V(G)$ and $i\in [4]$, if $a\prec_{1}b$ and $b\prec_{i}c$, then $ac\in E(G)$.

(3) For any $a,b,c\in V(G)$ and $i\in [4]$, if $a\prec_{i}b$ and $b\prec_{2}c$, then $ac\in E(G)$.

\begin{theorem}\label{dmagicalupperbound}
	Let $G$ be a double-magical graph. If $\omega(G)=k$, then $\chi(G)\leq \frac{k+1}{2}\binom{k+2}{3}$.
\end{theorem}

\begin{proof}
Let $<_{1},<_{2},<_{3}$ be total orders on $V(G)$ witnessing $G$, and let $\prec_{1},\prec_{2},\prec_{3},\prec_{4}$ denote the partial orders defined above. Clearly, there is no chain of length $k+1$ with respect to any of the partial orders $\prec_{i}$, because that would contradict the assumption $\omega(G)=k$.
	
For $h=1,\dots,k$, let $S_{h}$ denote the set of vertices $v\in V(G)$ for which the size of a longest $\prec_{1}$-chain with maximal element $v$ is $k-h+1$. Then the sets $S_{1},\dots, S_{k}$ form a partition of $V(G)$, where each $S_h$ is a $\prec_{1}$-antichain that contains no clique of size $h+1$. Indeed, suppose that $C\subset S_{h}$ induces a clique of size $h+1$ in $G$, and consider the smallest vertex $v\in C$ with respect to the order $<_{1}$. There exists a $\prec_{1}$-chain $D$ of size $k-h+1$ ending at $v$. This implies that for every $a\in D$ and $b\in C$, we have $a\prec_{1}v$ and $v\prec_{i}b$ for some $i\in\{2,3,4\}$. Then, by (2), we would have $ab\in E(G)$. Hence, $D\cup C$ would induce a clique of size $k+1$, contradiction.

For $h=1,\dots,k$ and $m=1,...,h$, let $S_{h,m}$ denote the set of vertices in $S_{h}$ for which the largest $\prec_{2}$-chain in $S_{h}$ with smallest element $v$ has size $h-m+1$. As $\omega(G[S_{h}])\leq h$, the sets $S_{h,1},\dots,S_{h,h}$ are $\prec_{1}$- and $\prec_{2}$-antichains partitioning $S_{l}$. Further, $S_{h,m}$ contains no clique of size $m+1$. Otherwise, if $C\subset S_{h,m}$ forms a clique of size $m+1$ in $G$, then consider the largest vertex $v\in C$ with respect to the order $<_{1}$. There exists a $\prec_{2}$-chain $D$ of size $h-m+1$ whose smallest element is $v$. Hence, for every $a\in C$ and $b\in D$, we have $a\prec_{i}v$ and $v\prec_{2}b$ for some $i\in\{3,4\}$, which implies, by (3), that $ab\in E(G)$. Hence, $C\cup D$ would induce a clique of size $h+1$ in $S_{h}$, contradiction.

Thus, we obtained that  $S_{h,m}$ is a $\prec_{1}$- and $\prec_{2}$-antichain, which does not contain a clique of size $m+1$. In particular, the size of the longest $\prec_{3}$- and $\prec_{4}$-chains in $S_{l,m}$ is at most $m$. This means that $G[S_{h,m}]$ can be properly colored with $m^{2}$ colors. Indeed, set the color of $v\in S_{h,m}$ to be $\phi(v)=(r,q)$, where $r$ is the size of the largest $\prec_{3}$-chain with smallest element $v$, and  $q$ is the size of the largest $\prec_{4}$-chain with smallest element $v$. Then $\phi:S_{h,m}\rightarrow [m]^{2}$ is a proper coloring of $G[S_{h,m}]$.

  As $S_{h}=\bigcup_{m=1}^{h}S_{h,m}$, we have
   $$\chi(G[S_{h}])\leq \sum_{m=1}^{h}\chi(G[S_{h,m}])\leq \sum_{m=1}^{h}m^{2}=\frac{h(h+1)(2h+1)}{6}.$$
   Finally, since $V(G)=\bigcup_{h=1}^{k}S_{h}$, we obtain

   $$\chi(G)\leq \sum_{h=1}^{k}\chi(G[S_{h}])\leq\sum_{h=1}^{k}\frac{h(h+1)(2h+1)}{6}=\frac{k+1}{2}\binom{k+2}{3}.$$
\end{proof}

\section{Construction of double-magical graphs--Proof of Theorem~\ref{construction2}}\label{vertical:lower}
  In view of \ref{dmagical2}, to prove Theorem~\ref{construction2}, it is enough to construct a double-magical graph with the desired clique and chromatic numbers.

\begin{theorem}\label{maindmagical}
	For every positive integer $k\geq 2$, there exists a double-magical graph $G$ satisfying $\omega(G)=k$ and $\chi(G)=\frac{k+1}{2}\binom{k+2}{3}$.
\end{theorem}

In the rest of this section, we prove this theorem. The proof of Lemma \ref{dmagicalupperbound} reveals a lot about the structure of double-magical graphs satisfying the properties of Theorem \ref{maindmagical}, if they exist. To construct them, we use reverse engineering.

For any vector $\mathbf{v}\in \mathbb{R}^{d}$  and any $j\in[d]$, let $\mathbf{v}(j)$ denote the $j$th coordinate of $\mathbf{v}$. The \emph{sign vector} of $\mathbf{v}\in \mathbb{R}^{d}$ is the $d$-dimensional vector $\sg(\mathbf{v})$ with $$\sg(\mathbf{v})(i)=\begin{cases}
1 &\mbox{ if } \mathbf{v}(i)>0,\\
-1 &\mbox{ if } \mathbf{v}(i)<0,\\
0 &\mbox{ if } \mathbf{v}(i)=0.
\end{cases}$$

Let $\mathbf{v}_{1}=(1,1,1)$, $\mathbf{v}_{2}=(1,-1,-1)$, $\mathbf{v}_{3}=(1,1,-1)$ and $\mathbf{v}_{4}=(1,-1,1)$. For any $\mathbf{i}\in [k]^{4}$, let
$$P(\mathbf{i})=k^{3}\mathbf{i}(1)\mathbf{v}_{1}+k^{2}\mathbf{i}(2)\mathbf{v}_{2}+k\mathbf{i}(3)\mathbf{v}_{3}+\mathbf{i}(4)\mathbf{v}_{4}.$$
These $k^4$ points have the useful property that if $\mathbf{i}\neq\mathbf{i}'$, then the relative position of $P(\mathbf{i})$ and $P(\mathbf{i}')$ depends only on the smallest coordinate in which $\mathbf{i}$ and $\mathbf{i}'$ differ. We refer to this property as the \emph{LEX property} (short for ``lexicographic''), which is formally defined as follows.

\bigskip
\noindent
\textbf{LEX property}: Let $\mathbf{i},\mathbf{i}'\in [k]^{4}$ such that $\mathbf{i}\neq\mathbf{i}'$, and let $r$ be the smallest index such that $\mathbf{i}(r)\neq \mathbf{i}'(r)$. If $\mathbf{i}(r)>\mathbf{i}'(r)$, then
$$\sg(P(\mathbf{i})-P(\mathbf{i}'))=\mathbf{v}_{r}.$$

\bigskip

Let $$S=\{\mathbf{i}\in [k]^{4}\, :\, \mathbf{i}(1)+\mathbf{i}(2)\leq k+1,\; \mathbf{i}(2)\geq \mathbf{i}(3), \mbox{ and }\mathbf{i}(2)\geq \mathbf{i}(4)\},$$ so that we have
   $$|S|=\sum_{i=1}^{k}(k+1-i)i^2=\frac{k+1}{2}\binom{k+2}{3}.$$

 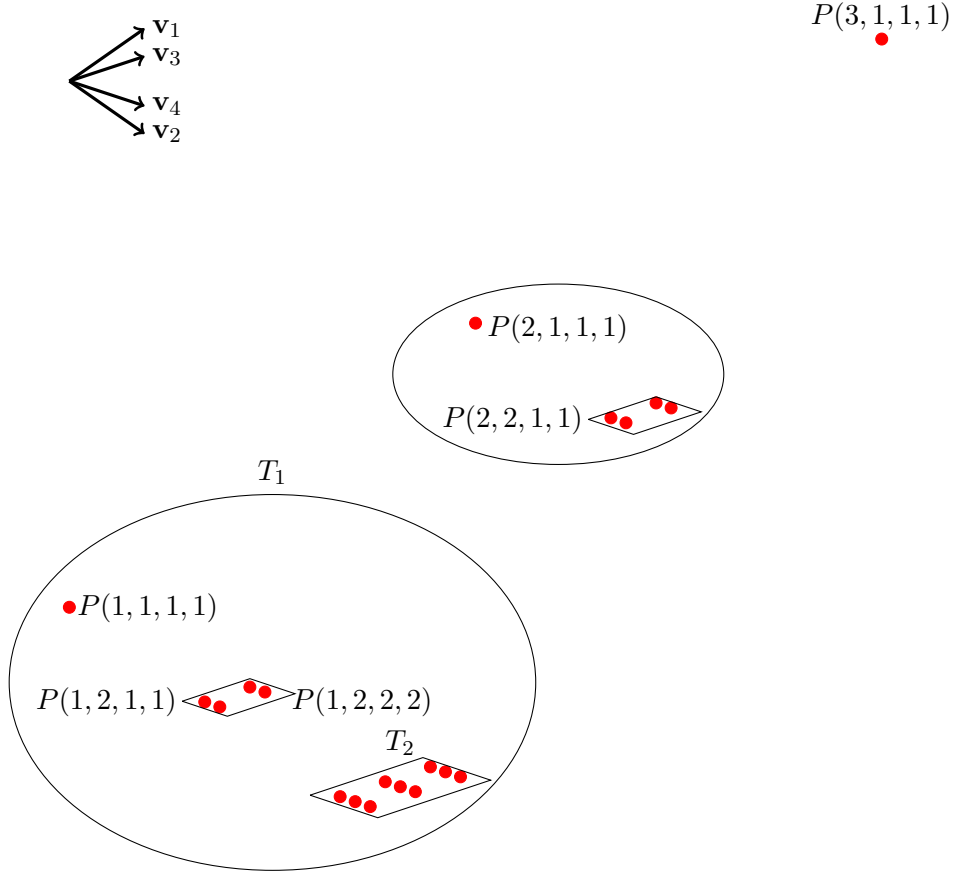
\begin{figure}[t]
 	\begin{center}
 		\pgfmathsetmacro{\N}{2}
 		\pgfmathsetmacro{\M}{3}
 		\pgfmathsetmacro{\xaa}{1}
 		\pgfmathsetmacro{\xab}{0.7}
 		\pgfmathsetmacro{\xba}{1}
 		\pgfmathsetmacro{\xbb}{-0.7}
 		\pgfmathsetmacro{\xca}{1}
 		\pgfmathsetmacro{\xcb}{0.33}
 		\pgfmathsetmacro{\xda}{1}
 		\pgfmathsetmacro{\xdb}{-0.33}
 		\pgfmathsetmacro{\s}{0.2}
 		\pgfmathsetmacro{\xv}{0}
 		\pgfmathsetmacro{\yv}{7}
 		\begin{tikzpicture}[scale=1]
 		
 	\draw[->, very thick] (\xv,\yv) -- (\xv+\xaa,\yv+\xab) ;
 	\draw[->, very thick] (\xv,\yv) -- (\xv+\xba,\yv+\xbb) ;
 	\draw[->, very thick] (\xv,\yv) -- (\xv+\xca,\yv+\xcb) ;
 	\draw[->, very thick] (\xv,\yv) -- (\xv+\xda,\yv+\xdb) ;

 	\node[] at (\xv+\xaa+0.3,\yv+\xab) {$\mathbf{v}_{1}$} ;
 	\node[] at (\xv+\xba+0.3,\yv+\xbb) {$\mathbf{v}_{2}$} ;
 	\node[] at (\xv+\xca+0.3,\yv+\xcb) {$\mathbf{v}_{3}$} ;
 	\node[] at (\xv+\xda+0.3,\yv+\xdb) {$\mathbf{v}_{4}$} ;

 		\foreach \i in {0,...,\N}
 		{
 			\foreach \j in {0,...,\N}
 			{
 				\foreach \k in {0,...,\j}
 				{
 					\foreach \l in {0,...,\j}
 					{
 						\pgfmathtruncatemacro\result{\i+\j-1}
 						\ifthenelse{\result<\N}{\node[vertex,red,scale=1.5] at (\s*\M*\M*\M*\i*\xaa+\s*\M*\M*\j*\xba+\s*\M*\k*\xca+\s*\l*\xda,\s*\M*\M*\M*\i*\xab+\s*\M*\M*\j*\xbb+\s*\M*\k*\xcb+\s*\l*\xdb) {};}{;}			
 					%	\ifthenelse{\result>4}{\node[] at (\i+\j/4-0.2,\j+\i/4+0.35) {$A_{\i,\j}$} ;}{;}
 					}
 				}
 			}
 			
 		}
 	
 	\draw (2.7,-1) ellipse (3.5 and 2.5);
 	\node[] at (2.7,1.8) {$T_{1}$};
 %	\draw (4.4,-2.3) ellipse (1 and 0.75);
    \draw (3.2,-2.5) -- (4.1,-2.8) -- (5.6,-2.3) -- (4.7,-2) -- (3.2,-2.5) ;
 	\node[] at (4.4,-1.8) {$T_{2}$} ;
 	\draw (6.5,3.1) ellipse (2.2 and 1.2);
 	\draw (1.5,-1.25) -- (2.1,-1.45) -- (3,-1.15) -- (2.4,-0.95) -- (1.5,-1.25) ;
 	\draw (6.9,2.5) -- (7.5,2.3) -- (8.4,2.6) -- (7.8,2.8) -- (6.9,2.5) ;
 %	\draw (\s*\M*\M*\M*2*\xaa,\s*\M*\M*\M*2*\xab) circle (0.5) ;
 	
 	\node[] at (\s*\M*\M*\M*2*\xaa,\s*\M*\M*\M*2*\xab+0.3) {$P(3,1,1,1)$} ;
 	\node[] at (1.05,0) {$P(1,1,1,1)$} ;
 	\node[] at (0.5,-1.26) {$P(1,2,1,1)$} ;	
 	\node[] at (3.9,-1.26) {$P(1,2,2,2)$} ;	
 	\node[] at (6.5,3.7) {$P(2,1,1,1)$} ;
 	\node[] at (5.9,2.5) {$P(2,2,1,1)$} ;
 %	\node[] at (9.2,2.6) {$P(2,2,2,2)$} ;
 		
 		\end{tikzpicture}
 		\caption{An illustration of the points $P(\mathbf{i})$ for $\mathbf{i}\in S$, $k=3$.}
 		\label{figure5}
 	\end{center}
 \end{figure}

An ordered triple of points $(u,v,w)\in\mathbb{R}^{3}\times\mathbb{R}^{3}\times\mathbb{R}^{3}$ is called a \emph{hole} if $u(1)<v(1)<w(1)$, and either $v(2)<\min\{u(2),w(2)\}$, or $v(3)<\min\{u(3),w(3)\}$.

\begin{claim}\label{hole_upperbound}
	Let $H\subset S$. If the set $\{P(\mathbf{i}):\mathbf{i}\in H\}$ does not contain a hole, then $|H|\leq k$.
\end{claim}

\begin{proof}
	Let $S=S_{k}$. We prove this claim by induction on $k$. If $k=1$, $S$ contains one element, so there is nothing to prove.
	
	Suppose that $k\geq 2$. Let $T_{1}=\{\mathbf{i}\in S:\mathbf{i}(1)=1\}$ and $H_{1}=H\cap T_{1}$. (See Figure \ref{figure5} for an illustration.) We distinguish two cases.
	
	{\em Case 1:} $|H_{1}|\leq 1$. Define $$H'=\{(i_{1}-1,i_{2},i_{3},i_{4}):(i_{1},i_{2},i_{3},i_{4})\in H\setminus H_{1}\}.$$
Then  $H'\subset S_{k-1}$ and $H'$ does not contain a hole. Hence, we obtain $|H'|\leq k-1$, by the induction hypothesis. On the other hand, $|H'|\geq |H|-1$, which yields that $|H|\leq k$.
	
	{\em Case 2:} $|H_{1}|\geq 2$. In this case, we must have $H=H_{1}$. Otherwise, choose $\mathbf{i},\mathbf{i}'\in H_{1}$, $\mathbf{j}\in H\setminus H_{1}$, and let $u=P(\mathbf{i}),v=P(\mathbf{i}')$, and $w=P(\mathbf{j})$. Suppose without loss of generality that $u(1)<v(1)$. Then $u(1)<v(1)<w(1)$, and by the LEX property we have $w(2)\geq \max\{u(2),v(2)\}$ and $w(3)\geq\max\{u(3),v(3)\}$. Therefore, if $(u,v,w)$ is not a hole, then we must have $u(2)<v(2)<w(2)$ and $u(3)<v(3)<w(3)$. However, this means that $\sg(v-u)=(1,1,1)=\mathbf{v}_{1}$, which contradicts the LEX property, as $\mathbf{i}(1)=\mathbf{i}'(1)$.
	
	Hence, we can suppose that $H=H_{1}\subset T_{1}$. Let $T_{2}=\{\mathbf{i}\in S:\mathbf{i}(1)=1, \mathbf{i}(2)=k\}\subset T_{1}$ and $H_{2}=H\cap T_{2}$. Again, we distinguish two subcases.
	
	{\em Subcase 1:} $|H_{2}|\leq 1$. Define $H'=H\setminus H_{2}$. Then $H'\subset S_{k-1}$ and $H'$ does not contain a hole, which yields, by the induction hypothesis, that $|H'|\leq k-1$. On the other hand, $|H'|\geq |H|-1$, so $|H|\leq k$.
	
	{\em Subcase 2:} $|H_{2}|\geq 2$. In this case, we show that $H=H_{2}$. Otherwise, let $\mathbf{i},\mathbf{i}'\in H_{1}$, $\mathbf{j}\in H\setminus H_{2}$, and $u=P(\mathbf{j}),v=P(\mathbf{i})$ and $w=P(\mathbf{i}')$. Suppose without loss of generality that $v(1)<w(1)$. Then $u(1)<v(1)<w(1)$, $u(2)\geq \max\{v(2),w(2)\}$, and $u(3)\geq\max\{v(3),w(3)\}$, by the LEX property. Thus, $(u,v,w)$ is a hole, unless $u(2)>v(2)>w(2)$ and $u(3)>v(3)>w(3)$, which would mean that the $\sg(w-v)=(1,-1,-1)=\mathbf{v}_{2}$. However, this contradicts the LEX property, because $\mathbf{i}(2)=\mathbf{i}'(2)$.
	
	Hence, we can suppose that $H=H_{2}\subset T_{2}$. Here, $T_{2}$ is partitioned into $k$ sets $U_{1},\dots,U_{k}$, where $U_{l}=\{(1,k,l,m):m=1,\dots,k\}$ for $l=1,\dots,k$. Note that $|U_{l}|=k$. We show that $H$ is either completely contained in one of the sets $U_{l}$, or $H$ intersects each of $U_{1},\dots,U_{k}$ in at most one element. In either case, we get $|H|\leq k$. Suppose to the contrary that there exists $l\neq l'$ and three elements $\mathbf{i},\mathbf{i}'\in U_{l}\cap H$, $\mathbf{j}\in U_{l'}\cap H$. Let $u=P(\mathbf{i})$, $v=P(\mathbf{i}')$, and $w=P(\mathbf{j})$. Without loss of generality, suppose that $u(1)<v(1)$. Now there are two cases depending on the order of $l$ and $l'$. If $l<l'$, then by the LEX property $u(1)<v(1)<w(1)$, $v(2)<u(2)<w(2)$, and $w(3)<u(3)<v(3)$, so $(u,v,w)$ is a hole. If $l'<l$, then $w(1)<u(1)<v(1)$, $w(2)<v(2)<u(2)$, and $u(3)<v(3)<w(3)$, so $(w,u,v)$ is a hole.		
\end{proof}

The rest of the proof of Theorem~\ref{maindmagical} is very similar to that of the proof of Theorem~\ref{mainmagical}.
First, we set a few parameters, to simplify the discussion. Let $t=24k^{4}\log k$, $\lambda=1/k^{4}$, $h=t^{k^{4}}k^{4k^{4}+16}$, $n=6h$, and $p=t/n$.

For each $\mathbf{i}\in S$, let $A_{\mathbf{i}}$ be a set of $n$ arbitrary points with distinct coordinates, whose distances from $P(\mathbf{i})$ are smaller than $1/2$. Let $V=\bigcup_{\mathbf{i}\in S} A_{\mathbf{i}}$. The main property of the sets $A_{\mathbf{i}}$ that we need is that for any $\mathbf{i},\mathbf{i}'\in S$ such that $\mathbf{i}\neq\mathbf{i}'$, and for every $u\in A_{\mathbf{i}}$ and $v\in A_{\mathbf{i}'}$, we have $\sg(u-v)=\sg(P(\mathbf{i})-P(\mathbf{i}'))$. In other words, the relative position of $u$ and $v$ only depends on $\mathbf{i}$ and $\mathbf{i}'$.

 Let $<_{1},<_{2},<_{3}$ be the three total orderings on $V$ given by the order of the $x,y,z$-coordinates of the points of $V$, respectively.

 As in Section~\ref{grounded:lower}, we call a pair of vertices  $\{u,v\}$ in $V$ \emph{available} if $u\in A_{\mathbf{i}}$, $v\in A_{\mathbf{i}'}$ and $\mathbf{i}\neq\mathbf{i}'$.

Let $G_{0}$ be the graph on $V$, in which each available pair of vertices is joined by an edge independently with probability $p$. Let $G^1=G^{1}_{<_{1},<_{2}}$ and $G^2=G^{2}_{<_{1},<_{3}}$ be the minimal magical graphs containing the edges of $G_{0}$, respectively, and let $G'$ be the graph on vertex set $V$ with $E(G')=E(G^{1})\cap E(G^{2})$. Then $G'_{<_{1},<_{2},<_{3}}$ is a double-magical graph.

\begin{claim}\label{independent:dmagical}
	 With probability at least $2/3$, the graph $G'$ has no independent set larger than $(1+\lambda)n$.
\end{claim}

\begin{proof}
	Repeating the same argument as in the proof of Claim \ref{independent}, we obtain that the probability that $V$ has a $(1+\lambda)n$-element independent subset is at most
	$$\binom{|V|}{(1+\lambda)n}(1-p)^{\lambda n^{2}/2}<(ek^{4})^{(1+\lambda)n}e^{-t\lambda n/2}=e^{(1+4\log k)(1+\lambda)n-t\lambda n/2}<1/3.$$
\end{proof}

Next, we bound the expected number of holes that form a triangle in $G'$.

\begin{claim}
	Let $N$ denote the number of holes in $V$ that induce a triangle in $G'$. Then we have $\mathbb{E}(N)<h$.
\end{claim}

\begin{proof}
	We proceed just like we did in the proof of Claim \ref{holes}. Let $(u,v,w)$ be a hole in $V$. We need to upper bound the probability that $(u,v,w)$ induces a triangle in $G'$. If $uv,vw,uw\in E(G')$, then $uv,vw,uw\in E(G^{1})\cap E(G^{2})$. As $uv,vw,uw\in E(G^{1})$, there exist three mountain-paths $P_{uv}^{1},P^{1}_{vw},P^{1}_{uw}$ in $G_{<_{1},<_{2}}$ with endpoints $\{u,v\}$, $\{v,w\}$, and $\{u,w\}$, respectively (see the definition above Lemma~\ref{mountainpaths}). As $uv,vw,uw\in E(G^{2})$, there exist three mountain-paths $P_{uv}^{2},P^{2}_{vw},P^{2}_{uw}$ in $G_{<_{1},<_{3}}$ with endpoints $\{u,v\}$, $\{v,w\}$, and $\{u,w\}$, respectively. Note that   $V(P_{uv}^{1})\cap V(P_{vw}^{1})=V(P_{uv}^{2})\cap V(P_{vw}^{2})=\{v\}$.
	
	Consider the graph $P=P_{uv}^{1}\cup P_{vw}^{1}\cup P_{uw}^{1}\cup P_{uv}^{2}\cup P_{vw}^{2}\cup P_{uw}^{2}$. This graph is connected, but it is not a tree. Indeed, there are two different paths between $u$ and $w$. This is true, because  $(u,v,w)$ is a hole, so we have either $v<_{2}u$ and $v<_{2}w$, or $v<_{3}u$ and $v<_{3}w$. If $v<_{2}u$ and $v<_{2}w$, then $P_{uw}^{1}$ does not contain $v$, so $P_{uw}^{1}$ and $P_{uv}^{1}\cup P_{vw}^{1}$ are two distinct paths between $u$ and $w$. Analogously, if $v<_{3}u$ and $v<_{3}w$, then $P_{uw}^{2}$ does not contain $v$, so $P_{uw}^{2}$ and $P_{uv}^{2}\cup P_{vw}^{2}$ are two distinct paths between $u$ and $w$.
	
	Since $P$ is a connected graph which is not a tree, we have $|E(P)|\geq |V(P)|$. From this point, we can mimic the calculations from Claim \ref{holes}.
	
	Let $\mathcal{P}$ be the set all graphs $P$ with the above property which appear in $G_{0}$ with positive probability. Then
	\begin{align*}
	\mathbb{P}(\{u,v,w\}\mbox{ induces a triangle in }G')&=\mathbb{P}(P\mbox{ is a subgraph of }G_{0}\mbox{ for some }P\in\mathcal{P})\\
	&\leq \sum_{P\in\mathcal{P}}\mathbb{P}(P\mbox{ is a subgraph of }G_{0}).
	\end{align*}
	
	For every $P\in\mathcal{P}$, any edge of $P$ is present in $G_{0}$ independently with probability $p$, so the probability that $P$ is a subgraph of $G_{0}$ is $p^{|E(P)|}$, which is at most $p^{|V(P)|}$. The number of graphs in $\mathcal{P}$ with exactly $m$ vertices is at most $\binom{|V|}{m-3}<(k^{4}n)^{m-3}$, as each member of $\mathcal{P}$ contains the three vertices $u,v,w$. Finally, every member of $\mathcal{P}$ has at most $|S|\leq k^{4}$ vertices, so that we can write
	
	$$\sum_{P\in\mathcal{P}}\mathbb{P}(P\mbox{ is a subgraph of }G_{0})\leq \sum_{m=3}^{k^{4}}p^{m}(k^{4}n)^{m-3}<t^{k^{4}}k^{4k^{4}+4}n^{-3}.$$
	As the number of holes in $V$ is at most $\binom{|V|}{3}<|V|^{3}<k^{12}n^{3}$, we obtain
	$$\mathbb{E}(N)<t^{k^{4}}k^{4k^{4}+16}=h.$$
\end{proof}

Applying Markov's inequality, the probability that $V$ contains more than $3h$ holes that induce a triangle in $G'$ is at most $1/3$. This means that there exists a magical graph $G'$ on $V$ such that $G'$ has no independent set of size $(1+\lambda)n$, and $G$ contains at most $3h$ triangles whose vertices form a hole. By deleting a vertex of each such hole in $G'$, we get a magical graph $G$ with at least $|S|n-3h$ vertices, which has no triangle that forms a hole, and no independent set of size larger than $(1+\lambda)n$.

We show  that $\chi(G)\geq |S|=\frac{k+1}{2}\binom{k+2}{3}$. Otherwise, if $\chi(G)\leq |S|-1$, then $G$ contains an independent set of size
$$\frac{|V(G)|}{|S|-1}\geq \frac{|S|n-3h}{|S|-1}=\left(1+\frac{1}{|S|-1}\right)n-\frac{3h}{|S|-1}>(1+\lambda)n,$$
a contradiction.

It remains to prove that $G$ has no clique of size $k+1$. Suppose that $C\subset V$ is a clique in $G$. Then $C$ intersects each $A_{\mathbf{i}}$ in at most one vertex for $\mathbf{i}\in S$, and $C$ does not contain a hole. Let $K=\{\mathbf{i}\subset S:A_{\mathbf{i}}\cap C\neq \emptyset\}$. The condition that $C$ does not contain a hole implies that $K$ does not contain a hole. But then, by Lemma \ref{hole_upperbound}, we have $|C|=|K|\leq k$. This completes the proof of Theorem \ref{maindmagical}.

\section{Concluding remarks}\label{sect:remarks}

We proved that best $\chi$-bounding function for the family of disjointness graphs of $x$-monotone curves satisfies $f(k)=\Theta(k^{4})$. After the main results presented in this paper, it seems reasonable to ask that what is the precise value of $f$.

\begin{problem}
	Let $f(k)$ denote the smallest $m$ such that for any collection $\mathcal{C}$ of $x$-monotone curves, if the disjointness graph $G$ of $\mathcal{C}$ satisfies $\omega(G)=k$, then $\chi(G)\leq m$. Determine $f(k)$.
\end{problem}

The results of our paper are partially motivated by the problem of Larman et al. \cite{LMPT} discussed in Section \ref{sect:grounded}. That is, what can we say about the order of the function $g(n)$, where $g(n)$ denotes the maximal $m$ such that every collection of $n$ convex sets contains either $m$ pairwise intersecting, or $m$ pairwise disjoint elements. One way to approach this problem would be to consider the corresponding question for magical graphs.

\begin{problem}
	Let $h(n)$ denote the maximal $m$ such that any magical graph on $n$ vertices contains either a clique or an independent set of size $m$. Determine $h(n)$.
\end{problem}

We have $h(n)\gtrapprox 1.26n^{1/3}$ by Theorem \ref{upperbound}. By the argument used for the proof of Theorem \ref{cor2}, we obtain that if $h(n)\geq n^{\alpha}$ holds, then $g(n)\geq n^{\alpha/(2\alpha+1)}$. Any improvement over the best known lower bound on $h(n)$ would yield a better lower bound on $g(n)$ than the roughly $n^{1/5}$ bound in \cite{LMPT}; see also Corollary~\ref{corcor}.

\begin{figure}[t]
	\begin{center}
		\begin{tikzpicture}[scale=1]
		
		\node[vertex, label=above:$1$] (A1) at (0,0) {};
		\node[vertex, label=below:$2$] (A2) at (1,0) {};
		\node[vertex, label=below:$3$] (A3) at (2,1) {};
		\node[vertex, label=below:$4$] (A4) at (3,1) {};
		\node[vertex, label=above:$5$] (A5) at (2,-1) {};
		\node[vertex, label=above:$6$] (A6) at (3,-1) {};
		\node[vertex, label=above:$7$] (A7) at (4,0) {};
		\node[vertex, label=below:$8$] (A8) at (5,0) {};
		
		\draw (A1) -- (A2) -- (A3) -- (A4) -- (A7) -- (A8) ;
		\draw (A2) -- (A5) -- (A6) -- (A7) ;
		\draw (A2) -- (A7) ;
		\draw (A1) -- (A5) -- (A7) ; \draw (A2) -- (A4) -- (A8) ;
		\draw (A1) to[out=-90,in=-90, distance=2cm] (A7) ;
		\draw (A2) to[out=90,in=90, distance=2cm] (A8) ;

		\end{tikzpicture}
		\caption{A semi-comparability graph $G_{<_{1}}$ for which there is no $<_{2}$ such that $G_{<_{1},<_{2}}$ is magical. The numbers of the vertices induce the ordering $<_{1}$.}
		\label{figure3}
	\end{center}
\end{figure}
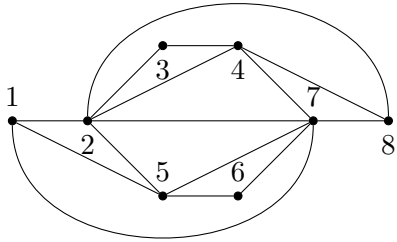
\bigskip

A \emph{0--1 curve} is a curve $C$ whose two endpoints lie on the vertical lines $\{x=0\}$ and $\{x=1\}$, and $C$ is contained in the strip $\{0\leq x\leq 1\}$. The disjointness graph of a collection of 0--1 curves is a comparability graph, and every comparability graph can be realized as the disjointness graph of some collection of 0--1 curves, see \cite{SiSiUr, Lo, PT2}.

We have seen in Lemma \ref{semiposet0} that the disjointness graph of grounded $x$-monotone curves is a semi-comparability graph. One might wonder if the converse is also true, that is, if every semi-comparability graph can be realized as the disjointness graph of grounded $x$-monotone curves. Or rather, if $G_{<_{1}}$ is a semi-comparability graph, then there exists a total ordering $<_{2}$ on $V(G)$ such that $G_{<_{1},<_{2}}$ is magical. Unfortunately, this is not true: a counterexample is presented in Figure \ref{figure3}.

Nevertheless, it seems that semi-comparability graphs capture many properties of disjointness graphs of grounded $x$-monotone curves.  In an upcoming work \cite{PT3}, we prove (among other results) the following property of semi-comparability graphs. If $G$ is a semi-comparability graph on $n$ vertices, then either $G$ contains a bi-clique of size $\Omega(n)$, or the complement of $G$ contains a bi-clique of size $\Omega(n/\log n)$. This property is known to hold \cite{FP} for the disjointness graph of an arbitrary family of $n$ curves, but its proof is highly geometric.


\begin{thebibliography}{9}
%\bibitem{AA} P.~K.~Agarwal, N.~Alon, B.~Aronov, B., and S.~Suri, {\em Can visibility graphs be represented compactly?}, Discrete Comput. Geom. {\bf 12} (1994): 347--365.

\bibitem{AgMu}  P.~K.~Agarwal and N.~H.~Mustafa, {\em Independent set of intersection graphs of convex objects in 2D,} {Comput. Geom.} {\bf 34} (2) (2006): 83--95.

\bibitem{AG}	
E. Asplund and B. Gr\"{u}nbaum,
\emph{On a colouring problem,}
 Math. Scand. {\bf 8} (1960): 181--188.	

\bibitem{Bra} P.~Brass, W.~Moser, and J.~Pach, {\em Research Problems in Discrete Geometry,} Springer, New York, 2005.

\bibitem{B}
J. P. Burling,
\emph{On Coloring Problems of Families of Prototypes (PhD thesis),}
University of Colorado, Boulder, 1965.

\bibitem{CaCa}  S.~Cabello, J.~Cardinal, and S.~Langerman, {\em The clique problem in ray intersection graphs,} Discrete Comput. Geom. {\bf 50} (3) (2013): 771--783.

\bibitem{Cer} M. R. Cerioli, L. Faria, T. O. Ferreira, and F. Protti, {\em On minimum clique partition and maximum independent set on unit disk graphs and penny graphs: complexity and approximation,} Electr. Notes Discr. Math. {\bf 18} (2004): 73--79.

%\bibitem{Chl} M. Chleb\'{\i}k and J. Chleb\'{\i}kova, {\em Approximation hardness of optimization problems in intersection graphs of $d$-dimensional boxes}, 16th Ann. ACM-SIAM Symp. Discrete Algorithms (SODA 2005), 267--276, 2005.

\bibitem{D}
R. P. Dilworth,
\emph{A decomposition theorem for partially ordered sets,}
Annals Math. {\bf 51} (2) (1950): 161--166.

\bibitem{Dum}  A.~Dumitrescu and J.~Pach, {\em Minimum clique partition in unit disk graphs,} Graphs Combin. {\bf 27} (3) (2011): 399--411.

\bibitem{Eid} S.~Eidenbenz and C.~Stamm, {\em MAXIMUM CLIQUE and MINIMUM CLIQUE PARTITION in Visibility Graphs,} IFIP TCS 2000: 200--212.
	
\bibitem{LMPT} D. Larman, J. Matou\v{s}ek, J. Pach, and J. T\"{o}r\H{o}csik,
\emph{ A Ramsey-type result for convex sets,}
Bull. London Math. Soc. {\bf 26} (1994): 132--136.

\bibitem{FP11} J.~Fox and J.~Pach, {\em Computing the independence number of intersection graphs}, in: {Proc. 22nd Ann. ACM-SIAM Symp. on Discrete Algorithms (SODA 2011)}, SIAM, Philadelphia, PA, 2011, 1161--1165.

\bibitem{FP}
J. Fox and J. Pach,
\emph{String graphs and incomparability graphs,}
Advances Math. {\bf 230} (2012): 1381--1401.

\bibitem{F}
Z. F\"{u}redi,
\emph{The maximum number of unit distances in a convex $n$-gon,}
J. Combin. Theory, Ser. A {\bf 55} (2) (1990): 316--320.


\bibitem{GaJ} M.~R.~Garey and D.~S.~Johnson, {\em Computers and Interactibility: A Guide to the Theory
of NP Completeness}, W.~H.~Freeman and Company, San Fransisco, 1979.

\bibitem{Gav} F. Gavril,
\emph{Algorithms for a maximum clique and a maximum independent set of a circle graph,}
Networks {\bf 3} (1973): 261--273.

%\bibitem{Gol} M.~C.~Golumbic, {\em Algorithmic Graph Theory and Perfect Graphs,} Academic Press, New York, 1980.

\bibitem{Gy}
A. Gy\'{a}rf\'as,
\emph{On the chromatic number of multiple interval graphs and overlap graphs,}
Discrete Math. {\bf 55} (2) (1985): 161--166. Corrigendum: Discrete Math. {\bf 62} (3) (1986): 333.

\bibitem{Gy1}
A. Gy\'{a}rf\'as, {\em Problems from the world surrounding perfect graphs,} Zastos. Mat. {\bf 19} (1987): 413--441.

%\bibitem{K} G. K\'arolyi, \emph{On point covers of parallel rectangles,} Periodica Mathematica Hungarica {\bf 23} (2) (1991): 105--107.

\bibitem{KeS} M.~Keila and L.~Stewart,
{\em Approximating the minimum clique cover and other hard problems in subtree filament graphs,}
Discrete Appl. Math. {\bf 154} (14) (2006): 1983--1995.

\bibitem{Chaya} C. Keller, S. Smorodinsky, and G. Tardos, {\em On Max-Clique for intersection graphs of sets and the Hadwiger-Debrunner numbers,} in: {Proc. 28th Annu. Symp. Discrete Algs. (SODA 2017)}, 2017, 2254 -- 2263.

%\bibitem{Kel} E.~Kellerman, {\em Determination of keyword conflict}, IBM Tech. Disclosure Bull. {\bf 16} (2) (1973): 544--546.

\bibitem{KT}
D. Kor\'andi and I. Tomon,
\emph{Improved Ramsey-type results in comparability graphs,}
arXiv:1810.00588 preprint.

\bibitem{Ko1} A.~Kostochka, {\em On upper bounds on the chromatic numbers of graphs,}
Transactions Inst. Math., Vol. {\bf 10}, Siberian Branch of the Acad. Sci. USSR (1988): 204--226 (in Russian).

\bibitem{Ko2} A.~Kostochka, {\em Coloring intersection graphs of geometric figures with a given clique number,} in: Towards a Theory of Geometric Graphs, Contemp. Math. {\bf 342}, Amer. Math. Soc., Providence, RI, 127--138, 2004.

\bibitem{KK} A.~Kostochka and J.~Kratochv\'{i}l, {\em Covering and coloring polygon-circle graphs,} Discrete Math. {\bf 163} (1-3) (1997): 299--305.

\bibitem{KoN} A.~Kostochka and J.~Ne\v set\v ril, {\em Chromatic number of geometric intersection graphs}, in: M.~Klazar (Ed.), 1995 Prague Midsummer Combinatorial Workshop, KAM Series {\bf 95–309}, Charles University, Prague (1995), 43--45.

\bibitem{KrM} J.~Kratochv\'{\i}l and J.~Matou\v sek, {\em Intersection graphs of segments,} {J. Combin. Theory Ser. B} {\bf 62} (2) (1994): 289--315.

\bibitem{KrNe}  J.~Kratochv\'{\i}l and J.~Ne\v set\v ril, {\em INDEPENDENT SET and CLIQUE problems in intersection-defined classes of graphs,} {Comment. Math. Univ. Carolin.} {\bf 31} (1) (1990): 85--93.

\bibitem{Ky} J. Kyn\v{c}l, \emph{Ramsey-type constructions for arrangements of segments,} European J. Combin. {\bf 33} (3) (2012): 336--339.

\bibitem{Lar}
D.~Larman, J.~Matou\v sek, J.~Pach, and J.~T\"or\H ocsik, {\em A Ramsey-type result for convex sets,} Bull. London Math. Soc. {\bf 26} (2) (1994): 132--136.

\bibitem{Las}  M.~Laso\'n, P.~Micek, A.~Pawlik, and B.~Walczak, {\em Coloring intersection graphs of arc-connected sets in the plane},  Discrete Comput. Geom. {\bf 52} (2) (2014): 399--415.

\bibitem{Lo}
L.~Lov\'asz, {\em Perfect graphs,} in: Selected Topics in Graph Theory, vol. {\bf 2}, Academic Press, London, 1983, 55--87.

\bibitem{M}
 S. McGuinness,
 \emph{Colouring arcwise connected sets in the plane I,}
Graphs and Combinatorics {\bf 16} (4) (2000): 429--439.

\bibitem{MWW}
T. M\"{u}tze, B. Walczak, and V. Wiechert,
\emph{Realization of shift graphs as disjointness graphs of 1-intersecting curves in the plane,}
arXiv:1802.09969

\bibitem{PT}
J. Pach and G. Tardos,
\emph{Forbidden paths and cycles in ordered graphs and matrices,}
Israel Journal of Mathematics {\bf 155} (2006): 359--380.

\bibitem{PTT}
J. Pach, G. Tardos, and G. T\'{o}th, \emph{Disjointness graphs of segments,} in: 33rd Internat. Symp.  Comput. Geom. (SoCG 2017), vol. {\bf 77} Leibniz Internat. Proc. Informatics (LIPIcs), 59:1--15, Leibniz-Zentrum f\"ur Informatik, Dagstuhl, 2017.

\bibitem{PT2}
J. Pach and G. T\'oth,
\emph{Comments on Fox News,}
Geombinatorics {\bf 15} (2006): 150--154.

\bibitem{PT3}
J. Pach and I. Tomon,
\emph{Large bi-cliques in the intersection graphs of $x$-monotone curves and ordered graphs,}
in preparation.

\bibitem{PaT}  J.~Pach and J.~T\"or\H ocsik, {\em Some geometric applications of Dilworth's theorem,} Discrete Comput. Geom. {\bf 12} (1) (1994): 1--7.

\bibitem{PKKLMTW}
A. Pawlik, J. Kozik, T. Krawczyk, M. Laso\'{n}, P. Micek, W. T. Trotter, and B. Walczak,
\emph{Triangle-free intersection graphs of line segments with large chromatic number,}
 J. Combin. Theory, Ser. B {\bf 105} (2014): 6--10.

%\bibitem{Pie} H.-P.~Piepho, {\em An algorithm for a letter-based representation of all-pairwise comparisons}, J. Comput. Graphical Statistics {\bf 13} (2) (2004), 456--466.

%\bibitem{RVM} S.~Rajagopalan, M.~Vachharajani, and S.~Malik, {\em Handling irregular ILP within conventional VLIW schedulers using artificial resource constraints}. In: Proc. CASES, ACM Press, 2000, 157--164.

\bibitem{Rok1}  A.~Rok and B.~Walczak, {\em Outerstring graphs are $\chi$-bounded,} in: 30th Internat. Symp. Comput. Geom. (SoCG'14), 136--143, ACM, New York, 2014.

\bibitem{Rok2}  A.~Rok and B.~Walczak, {\em Coloring curves that cross a fixed curve,} in: 33rd Internat. Symp. Comput. Geom., Art. No. 56, 15 pp., LIPIcs. Leibniz Int. Proc. Inform., 77, Schloss Dagstuhl, Leibniz-Zent. Inform., Wadern, 2017.

\bibitem{SiSiUr} J.~B.~Sidney, S.~J.~Sidney, and J.~Urrutia, {\em Circle orders, $n$-gon orders and the crossing number,} {Order} {\bf 5} (1) (1988): 1--10.

\bibitem{Suk} A.~Suk, {\em Coloring intersection graphs of $x$-monotone curves in the plane,} Combinatorica {\bf 34} (4) (2014): 487--505.

\bibitem{Sup}  K.~J.~Supowit, {\em Topics in Computational Geometry, PhD Thesis}, University of Illinois at Urbana-Champaign, Report UIUCDCS-R-81-1062, 1981.

\end{thebibliography}
\end{document}